\documentclass[12pt,oneside,english]{amsart}
\usepackage[T1]{fontenc}
\usepackage[latin9]{inputenc}
\usepackage{units}
\usepackage{amsthm}
\usepackage{amssymb}
\usepackage{setspace}
\usepackage{esint}
\makeatletter
\numberwithin{equation}{section}
\numberwithin{figure}{section}
 \theoremstyle{definition}
 \newtheorem*{defn*}{\protect\definitionname}
\theoremstyle{plain}
\newtheorem{thm}{\protect\theoremname}
  \theoremstyle{plain}
  \newtheorem{lem}[thm]{\protect\lemmaname}
  \theoremstyle{plain}
  \newtheorem{cor}[thm]{\protect\corollaryname}
  \newtheorem*{thm*}{Theorem}

\makeatother

\usepackage{babel}
  \providecommand{\corollaryname}{Corollary}
  \providecommand{\definitionname}{Definition}
  \providecommand{\lemmaname}{Lemma}
\providecommand{\theoremname}{Theorem}

\begin{document}
\begin{onehalfspace}
\global\long\def\ltnorm#1#2{\left\Vert #1\right\Vert _{L^{2}#2}}
\global\long\def\norm#1{\left\Vert #1\right\Vert }
\global\long\def\bbR{\mathbb{R}}
\global\long\def\bbC{\mathbb{C}}
\global\long\def\bbT{\mathbb{T}}
\global\long\def\bbZ{\mathbb{Z}}
\global\long\def\bbN{\mathbb{N}}
\global\long\def\eps{\varepsilon}
\global\long\def\rint#1#2{\underset{#1}{\overset{#2}{\int}}}
\global\long\def\SS{\mathcal{S}}

\title{Riesz sequences and generalized arithmetic progressions}

\author{Itay Londner}

\address{School of Mathematical Sciences, Tel-Aviv University, Tel-Aviv 69978,
Israel.}

\email{itaylond@post.tau.ac.il }

\begin{abstract}

The purpose of this note is to verify that the results attained in \cite{londner2014riesz} admit an extension to the multidimensional
setting. Namely, for subsets of the two dimensional torus we find
the sharp growth rate of the step(s) of a generalized arithmetic progression
in terms of its size which may be found in an exponential systems
satisfying the Riesz sequence property.

\end{abstract}

\maketitle

\section{Introduction}

We discuss a particular aspect of the classical interpolation problem
for functions supported on the two dimensional torus $\bbT^{2}=\left[-\pi,\pi\right)^{2}$.

\subsection{Interpolation sets and Riesz sequences}

Let $\mathcal{S}\subset\bbR^{2}$ be a bounded measurable set and
$\Lambda\subset\bbR^{2}$. We are interested in the Riesz sequence
property of the system of exponentials
\[
E\left(\Lambda\right):=\left\{ e^{i\left\langle \lambda,t\right\rangle }\right\} _{\lambda\in\Lambda}
\]
in the space $L^{2}\left(\mathcal{S}\right)$.

A system of vectors $\left\{ \varphi_{n}\right\} $ in a Hilbert space
$H$ is a \emph{Riesz sequence} (RS) if there exist constants $0<A\leq B$
such that 
\[
A\sum_{n}\left|c_{n}\right|^{2}\leq\norm{\sum_{n}c_{n}\varphi_{n}}^{2}\leq B\sum_{n}\left|c_{n}\right|^{2}\qquad\left(1\right)
\]
for every finite sequence of scalars $\left\{ c_{n}\right\} $.

The Riesz sequence property of the exponential system $E\left(\Lambda\right)$
in the space $L^{2}\left(\SS\right)$ (we abbreviate $\Lambda\in RS\left(\SS\right))$
can be reformulated in terms of the Paley-Wiener space $PW_{\SS}$
of all functions $f\in L^{2}\left(\bbR^{2}\right)$ whose Fourier
transform 
\[
\hat{f}\left(\xi\right)=\int f\left(t\right)e^{-i\left\langle \xi,t\right\rangle }dt
\]
is supported on $\SS$. 

A set $\Lambda$ is called an \emph{interpolation set} for $PW_{\SS}$
if for any possible ''data'' $\left\{ c_{\lambda}\right\} \in\ell^{2}\left(\Lambda\right)$
there exists at least one solution, i.e. a function $f\in PW_{\SS}$
satisfying $f\left(\lambda\right)=c_{\lambda}$ for all $\lambda\in\Lambda$.
It is well known that $\Lambda\in RS\left(\SS\right)$ if and only
if $\Lambda$ is an interpolation set for $PW_{\SS}$ (for details see \cite{young1980introduction}, \cite{olevskii2016functions}).

Usually the RHS inequality in $\left(1\right)$ is easier to prove
and mild assumptions on $\Lambda$ such as uniform discreteness suffices
(see \cite{young1980introduction}, \cite{olevskii2016functions}), while the LHS inequality is much more difficult.
In this paper we will restrict ourselves to the case where $\SS$
is a subset of the two dimensional torus $\bbT^{2}$ and $\Lambda$ is a subset of the integer
lattice $\bbZ^{2}$, a case in which the RHS inequality in $\left(1\right)$
always holds with $B=1$. A constant satisfying the LHS inequality of $\left(1\right)$
will be referred to as a \emph{lower Riesz bound} of the system $\left\{ \varphi_{n}\right\} $
in $H$.

\subsection{Generalized arithmetic progressions}

An arithmetic progression is defined by two parameters; the size of
a progression and its step. A central role in this paper is played
by the concept of generalized arithmetic progression. The definition
below is a common generalization to the concept of an arithmetic progression
(see \cite{MR1477155}).

\begin{defn*}
For a pair $w_{1},w_{2}\in\bbZ^{2}\backslash\left\{ 0\right\} $ of
linearly independent vectors and a pair of integers $d_{1},d_{2}>1$,
the \emph{rank $2$ generalized arithmetic progression} (GAP) of size
${\displaystyle d_{1}\cdot d_{2}}$ with steps $w_{1},w_{2}$ is the
set 
\[
GAP\left(w_{1},w_{2};d_{1},d_{2}\right)=\left\{ k_{1}w_{1}+k_{2}w_{2}\,|\,0\leq k_{j}\leq d_{j}-1\right\} .
\]
A \emph{rank 1 GAP} is the degenerated case when either $d_{1}$ or
$d_{2}$ equal $1$, then (when $d_{2}=1$) we denote $AP\left(w_{1};d_{1}\right)=GAP\left(w_{1},w_{2};d_{1},1\right)$.
\end{defn*}
\begin{defn*}
Given $w\in\bbZ^{2}\backslash\left\{ 0\right\} $ we say $v=\left(v_{1},v_{2}\right)\in\bbZ^{2}$
is the \emph{generator} of $w$ if $w\in v\bbN$ and $\gcd\left(v_{1},v_{2}\right)=1$.
We shall denote the set of all generators by $V$.
\end{defn*}
Note that every $w\in\bbZ^{2}\backslash\left\{ 0\right\} $ can be
uniquely represented $w=\ell v$, with $\ell\in\bbN$ and $v\in V$.

\subsection{Arithemtic progressions and the one dimensional torus}

	It is known, and can be proved directly, that every set $\SS\subset\bbT$ of positive admits a Riesz sequence which contains, for arbitrarily large $N$, an arithmetic progression of length $N$ and step $\ell$. But what can be said about the connection between $N$ and $\ell$?\\
	Given a set $\SS$ of positive measure it has been proved by Lawton (\cite{lawton2010minimal}, Cor. 2.1) that the following statments are equivalent:\\
	\emph{\begin{itemize}
			\item $E\left(\bbZ\right)$ can be decomposed into finitely many Riesz sequences in $L^2\left(\SS\right)$
			\item There exists a set $\Lambda\in RS\left(\SS\right)$ with bounded gaps
	\end{itemize}}

	The former was known as the Feichtinger conjecture for exponentials. The Feichtinger conjecture in its general form was proved equivalent to the Kadison-Singer problem (\cite{MR2277219}), which has been solved recently (see \cite{marcus2015interlacing}).\\
	It now follows that every set $\SS$ admits an exponential Riesz sequence $E\left(\Lambda\right)$ such that $\Lambda$ has positive lower density, which, in turn, determines that it must contain arbitrarily long AP. Note that by Gowers' quantitative version of Szemeredi's theorem (\cite{gowers2001new}), one can ensure a super exponential growth rate of the step. \\	
	In \cite{londner2014riesz}, the problem of determining an optimal growth rate of the step $\ell$ as a function of $N$ has been considered. Extending a result of Bownik and Speegle (\cite{MR2270922}, Thm 4.16), it was proved that a sublinear growth is not enough. \\
\begin{thm*}[\cite{londner2014riesz}, Thm. 1]
	There exists a set $\SS\subset\bbT$ of positive measure such that whenever a set $\Lambda\subset\bbZ$ contains, for arbitrarily large $N$, an arithmetic progression of length $N$ and step $\ell=O\left(N^{\alpha}\right)$, $\alpha<1$, the exponential system $E\left(\Lambda\right)$ is not a Riesz sequence in $L^{2} \left(\SS\right)$
\end{thm*}

	It was also proved that this rate is sh	arp, i.e., a linear growth already suffices.\\
	\begin{thm*}[\cite{londner2014riesz}, Thm. 2]
		Every set $\SS\subset\bbT$ of positive  measure admit an exponential Riesz sequence $E\left(\Lambda\right)$ in $L^{2} \left(\SS\right)$, such that the set $\Lambda$ contains arbitrarily long arithmetic progressions of length $N$ and step $\ell=O\left(N\right)$
	\end{thm*}

\subsection{Main problem and results}

In this paper we are interested in extending the results from \cite{londner2014riesz} to the two dimensional setting. Namely, for
sets $\SS\subset\bbT^{2}$ we consider subsets $\Lambda$ of the integer lattice containing arbitrarily large GAPs while keeping the Riesz sequence property for $E\left(\Lambda\right)$ in $L^2\left(\SS\right)$. We prove

\begin{thm}
For every $\eps>0$, there exists a set $S\subset\mathbb{T}^{2}$ with $\left|\SS\right|>1-\eps$ such that for any set
$\Lambda\subset\bbZ^{2}$ which contains, for arbitrarily large $N$,
a (translated) GAP of size $N^{2}$ and steps of length $O\left(N^{\alpha}\right)$,
$\alpha<1$, the exponential system $E\left(\Lambda\right)$ is not
a RS in $L^{2}\left(S\right)$. Moreover, the result holds if for
some fixed generators $v_{1},v_{2}\in V$, for infinitely many $N$'s
the set $\Lambda$ contains a (translated) GAP of size $N$ and steps
$w_{1},w_{2}$ of length $\left|w_{j}\right|\leq C\left(\alpha,v_{j}\right)N^{\alpha}$,
$\alpha<1$ for $j=1,2$ generated by $v_{1}$ and $v_{2}$ respectively.

\end{thm}

We also prove Theorem 1 is sharp.
\begin{thm}
Given a set $\mathcal{S}\subset\bbT^{2}$ of positive measure,
there is a set of frequencies $\Lambda\subset\bbZ^{2}$ such that
\begin{enumerate}
\item [(i)]For arbitrarily large $N$, the set $\Lambda$ contains
a GAP of size $N^{2}$ with step size of the order $O\left(N\right)$.
\item [(ii)] The exponential system $E\left(\Lambda\right)$ forms
a Riesz sequence in $L^{2}\left(\mathcal{S}\right)$. 
\end{enumerate}
\end{thm}

The basic approach to the proof is similar to that taken in \cite{londner2014riesz} with appropriate modifications araising from the two dimensional setting. In particular, using the terminology from section 1.2, in the one dimensional setting all arithmetic progressions have a single generator while the current result takes care of GAPs with many more possible generators simultaneously.\\

For $A,B\subset\bbR^{2}$, $x\in\bbR^{2}$ we let 
\[
A+B:=\left\{ \alpha+\beta\,|\,\alpha\in A\,,\,\beta\in B\right\} \;,\;x\cdot A:=\left\{ x\cdot\alpha\,|\,\alpha\in A\right\} \,.
\]

Throughout the text we use $\left|\SS\right|$ to denote the normalized
Lebesgue measure of the set $\SS$, also in various places we use
$c$ and $C$ as positive absolute constant which might be different
from one line to the next or even within the same line.

\section{Proof of theorem 1}

We use the following approach, for each GAP with eligible step(s)
we form a trigonometric polynomial having that spectrum. Then we subsequently
remove a subset of the torus on which most of its $L^{2}$ norm is
concentrated. The limiting set satisfies the desired properties.

\subsection{Construction of the set $\mathcal{S}$}

Recall the set $V$ defined in section $1.2$ and fix a partial ordering
of this set $V=\left\{ v_{m}\right\} _{m\in\bbN}$ so that
\[
\left|v_{m}\right|\leq\left|v_{m+1}\right|\quad\forall m.
\]
\begin{lem}

According to the fixed ordering of the set $V$, there exists an absolute
constant $c>0$ such that 
\[
cm^{\nicefrac{1}{2}}\leq\left|v_{m}\right|\;,\;\forall m
\]

\end{lem}

\begin{proof}

By definition a lattice point $v$ belongs to the set $V$ if and
only if its coordinates are co-prime. Now the lemma immediately follows
from the facts that the total number of lattice points inside a ball
of radius $r$ is $cr^{2}$, and the ratio of the number of points
of $V$ within that ball is approaching a positive constant as $r$
tends to infinity (\cite{apostol1976introduction}, Thm. 3.9).

\end{proof}
Fix $\eps>0$ and choose a decreasing sequence of positive numbers
$\left\{ \delta\left(n\right)\right\} $ such that 
\begin{enumerate}
\item $\sum_{n}\delta\left(n\right)<\left(\nicefrac{\eps}{2}\right)^{\nicefrac{1}{2}}$
\item $\delta\left(n\right)n^{\nicefrac{1}{\alpha}}\rightarrow\infty$ as
$n\rightarrow\infty$ for all $\alpha<1$.
\end{enumerate}

Let $w\in\mathbb{Z}^{2}\backslash\left\{ 0\right\} $ and find the
unique $\ell,m\in\mathbb{N}$ such that 
\[
w=\ell\cdot v_{m},
\]
we set 
\[
\rho_{w}=\delta\left(\ell\right)\delta\left(m\right).
\]

Next, take any $\xi\in\bbZ^{2}$ linearly independent of $w$. We
define an invariable linear transformation $L_{w}$ by 
\[
L_{w}e_{1}=w\,,\,L_{w}e_{2}=\xi
\]
where $e_{1},e_{2}\in\bbZ^{2}$ are the standard basis vectors. Consider
the sets 
\[
I_{w}=\left(-\rho_{w},\rho_{w}\right)\times\bbT\quad,\quad\tilde{I}_{w}=\bigcup_{u\in\bbZ^{2}}\left(I_{w}+2\pi u\right)
\]
and 
\[
I_{\left[w\right]}=\left(\left(L_{w}^{T}\right)^{-1}\left(\tilde{I}_{w}\right)\right)\cap\bbT^{2}.
\]
An important observation is that $I_{\left[w\right]}$ is independent
of the choice of $\xi$. We also define $I_{w}^{\prime}=\bbT\times\left(-\rho_{w},\rho_{w}\right)$
and $\tilde{I_{w}^{\prime}}$ respectively.

Since $\bbZ^{2}$ is a subgroup of $\left(L_{w}^{T}\right)^{-1}\left(\bbZ^{2}\right)$
(of index $\det L_{w}$), $\bbT^{2}$ is a union of $\det L_{w}$
disjoint copies of $\left(L_{w}^{T}\right)^{-1}\left(\bbT^{2}\right)$,
from which we can deduce that $I_{\left[w\right]}$ is a union of
$\det L_{w}$ disjoint translated copies of $\left(L_{w}^{T}\right)^{-1}\left(I_{w}\right)$,
therefore 
\[
\left|I_{\left[w\right]}\right|=\det L_{w}\left|\left(L_{w}^{T}\right)^{-1}\left(I_{w}\right)\right|=\left|I_{w}\right|=2\rho_{w}.
\]
We are now ready to define 
\[
{\displaystyle \mathcal{S}=\bbT^{2}\backslash\bigcup_{w\in\bbZ^{2}\backslash\left\{ 0\right\} }I_{\left[w\right]}}=\left(\bigcup_{w\in\bbZ^{2}\backslash\left\{ 0\right\} }I_{\left[w\right]}\right)^{c}
\]
whence 
\[
\left|\mathcal{S}\right|\geq1-\sum_{w\in\bbZ^{2}\backslash\left\{ 0\right\} }\left|I_{\left[w\right]}\right|=1-2\sum_{w\in\bbZ^{2}\backslash\left\{ 0\right\} }\rho\left(w\right)=1-2\sum_{\ell}\delta\left(\ell\right)\sum_{m}\delta\left(m\right)>1-\eps.
\]

\subsection{Proof for rank 1 GAP}

Let $w\in\mathbb{Z}^{2}$, $N\in\mathbb{N}$, and consider the trigonometric
polynomial with spectrum $AP\left(w;N^{2}\right)$
\[
P\left(x,y\right)=\sum_{j=1}^{N^{2}}c\left(j\right)e^{i\left\langle jw,\left(x,y\right)\right\rangle }.
\]
Since $\mathcal{S}\subseteq I_{\left[w\right]}^{c}$

\[
\underset{\mathcal{S}}{\int}\left|P\left(x,y\right)\right|^{2}\frac{dxdy}{\left(2\pi\right)^{2}}\leq\underset{I_{\left[w\right]}^{c}}{\int}\left|P\left(x,y\right)\right|^{2}\frac{dxdy}{\left(2\pi\right)^{2}}=
\]
taking $\left(\tilde{x},\tilde{y}\right)^{T}=L_{w}^{T}\left(x,y\right)^{T}$
we get 
\[
=\underset{\mathbb{T}^{2}\backslash I_{w}}{\int}\left|Q\left(\tilde{x},\tilde{y}\right)\right|^{2}\frac{d\tilde{x}d\tilde{y}}{\left(2\pi\right)^{2}}
\]
and ${\displaystyle Q\left(x,y\right)=\sum_{j=1}^{N^{2}}c\left(j\right)e^{ijx}}$
is the polynomial with spectrum $AP\left(e_{1};N^{2}\right)$. Set
$c\left(j\right)=\frac{1}{N}$ and note that for every $\left(x,y\right)\in\bbT^{2}$
we have $\left|Q\left(x,y\right)\right|\leq\frac{1}{N\sin\frac{x}{2}}$,
hence 
\[
\underset{\mathcal{S}}{\int}\left|P\left(x,y\right)\right|^{2}\frac{dxdy}{\left(2\pi\right)^{2}}\leq\frac{2}{N^{2}}\int_{\rho_{w}}^{\pi}\left|\sum_{j=1}^{N^{2}}e^{ij\tilde{x}}\right|^{2}\frac{d\tilde{x}}{2\pi}\leq\frac{C}{N^{2}}\int_{\rho_{w}}^{\pi}\frac{d\tilde{x}}{\sin^{2}\nicefrac{\tilde{x}}{2}}<\frac{C}{N^{2}}\int_{\rho_{w}}^{\pi}\frac{d\tilde{x}}{\tilde{x}^{2}}\leq\frac{C}{N^{2}\rho_{w}}
\]

Fix $\alpha<1$ and let $\Lambda\subset\bbZ^{2}$ be such that for
arbitrarily large $N$ one can find $w,\,M\in\bbZ^{2}$ (both depend
on $N$) for which $M+AP\left(w;N^{2}\right)\subset\Lambda$ and 
\[
\left|w\right|=\ell\left|v_{m}\right|\leq CN^{\alpha}.
\]

Applying Lemma 3 we get 
\[
\underset{\mathcal{S}}{\int}\left|P\left(x,y\right)\right|^{2}\frac{dxdy}{\left(2\pi\right)^{2}}\leq\frac{C}{N^{2}\rho_{w}}=\frac{C}{N^{2}\delta\left(\ell\right)\delta\left(m\right)}\leq
\]
\[
\leq\frac{C\left(\alpha\right)}{\ell^{\nicefrac{2}{\alpha}}\left|v_{m}\right|^{\nicefrac{2}{\alpha}}\delta\left(\ell\right)\delta\left(m\right)}\leq\frac{C\left(\alpha\right)}{\ell^{\nicefrac{2}{\alpha}}m^{\nicefrac{1}{\alpha}}\delta\left(\ell\right)\delta\left(m\right)}
\]
now by the choice of the sequence $\left\{ \delta\left(n\right)\right\} $
the last expression can be made arbitrarily small if either $\ell$
or $m$ are unbounded. The latter must be the case unless the step
$w$ remains bounded, but in this case the result still holds, since
$N$ can be made arbitrarily large.

For the moreover part, assume now that for some fixed $v\in V$ and
for arbitrarily large $N\in\bbN$\emph{ }we have $M+AP\left(\ell\cdot v,N\right)\subset\Lambda$,
where $M=M\left(\ell\right)\in\bbZ^{2}$ and $\ell\leq C\left(v\right)N^{\alpha}$.
In this case setting $w=\ell v$ and taking the coefficients $c\left(j\right)=\frac{1}{\sqrt{N}}$
gives
\[
\underset{\mathcal{S}}{\int}\left|P\left(x,y\right)\right|^{2}\frac{dxdy}{\left(2\pi\right)^{2}}\leq\frac{C}{N\rho_{w}}\leq\frac{C\left(\alpha,v\right)}{\ell^{\nicefrac{1}{\alpha}}\delta\left(\ell\right)}
\]
which can also be made arbitrarily small.

\subsection{Proof for rank 2 GAP}

Let $w_{1},w_{2}\in\mathbb{Z}^{2}$ linearly independent, and integers
$d_{1},d_{2}>1$. Consider $GAP\left(w_{1},w_{2};d_{1},d_{2}\right)$
and the trigonometric polynomial 
\[
{\displaystyle P\left(x,y\right)=\sum_{j_{1}=1}^{d_{1}}\sum_{j_{2}=1}^{d_{2}}c\left(j_{1},j_{2}\right)e^{i\left\langle j_{1}w_{1}+j_{2}w_{2},\left(x,y\right)\right\rangle }}.
\]

Define the invartiable linear transformation $L=L\left(w_{1},w_{2}\right)$
\[
Le_{j}=w_{j}\;,\;j=1,2.
\]
From the fact that $I_{\left[w\right]}$ depends solely on $w$, it
follows that 
\[
\left(L^{T}\right)^{-1}\left(\tilde{I}_{w_{1}}\cup\tilde{I^{\prime}}_{w_{2}}\right)\cap\bbT^{2}=I_{\left[w_{1}\right]}\cup I_{\left[w_{2}\right]},
\]
 which implies $\mathcal{S}\subseteq I_{\left[w_{1}\right]}^{c}\cap I_{\left[w_{2}\right]}^{c}$.\\
We estimate 
\[
\underset{\mathcal{S}}{\int}\left|P\left(x,y\right)\right|^{2}\frac{dxdy}{\left(2\pi\right)^{2}}\leq\underset{\left(L^{T}\right)^{-1}\left(\bbT^{2}\backslash\left(\tilde{I}_{w_{1}}\cup\tilde{I^{\prime}}_{w_{2}}\right)\right)}{\int}\left|P\left(x,y\right)\right|^{2}\frac{dxdy}{\left(2\pi\right)^{2}}=
\]
and after an appropriate change of variables we get 

\[
=\underset{\mathbb{T}^{2}\backslash\left(I_{w_{1}}\cup I_{w_{2}}^{\prime}\right)}{\int}\left|Q\left(\tilde{x},\tilde{y}\right)\right|^{2}\frac{d\tilde{x}d\tilde{y}}{\left(2\pi\right)^{2}}
\]
where 
\[
{\displaystyle Q\left(x,y\right)=\sum_{j_{1}=1}^{d_{1}}\sum_{j_{2}=1}^{d_{2}}c\left(j_{1},j_{2}\right)e^{i\left\langle \left(j_{1},j_{2}\right),\left(x,y\right)\right\rangle }}.
\]
Setting $N^{2}=d_{1}\cdot d_{2}$ and $c\left(j_{1},j_{2}\right)=\frac{1}{N}$
we get 
\[
\underset{\mathcal{S}}{\int}\left|P\left(x,y\right)\right|^{2}\frac{dxdy}{\left(2\pi\right)^{2}}\leq\frac{c}{N^{2}}\left(\int_{\rho_{w_{1}}}^{\pi}\left|\sum_{j_{1}=1}^{d_{1}}e^{ij\tilde{x}}\right|^{2}\frac{d\tilde{x}}{2\pi}\right)\left(\int_{\rho_{w_{2}}}^{\pi}\left|\sum_{j_{1}=1}^{d_{1}}e^{ij\tilde{y}}\right|^{2}\frac{d\tilde{y}}{2\pi}\right)<\frac{c}{N^{2}\rho_{w_{1}}\rho_{w_{2}}}\leq
\]
plugging our assumption $\left|w_{j}\right|=\ell_{j}\left|v_{m_{j}}\right|\leq CN^{\alpha}$
and the estimate from Lemma 3 
\[
\leq c\left(\alpha\right)\frac{1}{\delta\left(\ell_{1}\right)\delta\left(\ell_{2}\right)\left(\ell_{1}\ell_{2}\right)^{\nicefrac{2}{\alpha}}}\frac{1}{\delta\left(m_{1}\right)\delta\left(m_{2}\right)\left(m_{1}m_{2}\right)^{\nicefrac{1}{\alpha}}}
\]
The rest of the proof as well as the moreover part is similar to the
one in section 2.2 and hence omitted.

\section{Proof of Theorem 2}

Let $\mathcal{P}$ be the set of all prime numbers. Consider the vector
$w_{n,k}=\left(n,k\right)\in\bbZ^{2}$ and the collection of rank
1 GAPs
\[
B\left(n,k\right)=AP\left(w_{n,k};\,n^{2}\right)=\left\{ w_{n,k},\,2w_{n,k},\ldots,\,n^{2}w_{n,k}\right\} .
\]

\begin{lem}

The collection $\left\{ B\left(p,k\right)\right\} $ with $p\in\mathcal{P}$
and $k\in\left\{ 1,\ldots,p-1\right\} $ is pairwise disjoint.

\end{lem}

\begin{proof}

This simply follows from the fact that the lines $y=\frac{k}{p}x$,
$p\in\mathcal{P}$, $k\in\left\{ 1,\ldots,p-1\right\} $ all have
different slopes.

\end{proof}
\begin{lem}

Let $\left\{ a\left(m,m^{\prime}\right)\right\} _{m,m^{\prime}\in\bbZ}$
a sequence of non-negative numbers satisfying

\begin{enumerate}
\item ${\displaystyle \sum_{m,m^{\prime}\in\bbZ}a\left(m,m^{\prime}\right)\leq1}$.
\item $a\left(m,m^{\prime}\right)=a\left(-m,-m^{\prime}\right)$ for all
$m,m^{\prime}\in\bbZ$.
\end{enumerate}
Then for every $\eps>0$, there exist infinitely many $n\in\bbN$ so
that for each of which one can find $1\leq k\leq n-1$ s.t. 
\[
\sum_{\lambda\in B\left(n,k\right)}a\left(\lambda\right)<\frac{\eps}{n^{2}}\,.
\]

\end{lem}

\begin{proof}

By Lemma 4 we may write 
\[
\sum_{m,m^{\prime}\in\bbZ}a\left(m,m^{\prime}\right)\geq\sum_{p\in\mathcal{P}}\sum_{k=1}^{p-1}\sum_{\lambda\in B\left(p,k\right)}a\left(\lambda\right)\,.
\]
Assuming the contrary for some $\eps>0$, i.e., for all but finitely
many $p\in P$ we have ${\displaystyle \sum_{\lambda\in B\left(p,k\right)}a\left(\lambda\right)\geq\frac{\eps}{p^{2}}}$
for all $1\leq k\leq p-1$. It follows that 
\[
\sum_{k=1}^{p-1}\sum_{\lambda\in B\left(p,k\right)}a\left(\lambda\right)\geq\frac{\eps\left(p-1\right)}{p^{2}}>\frac{\eps}{3p}
\]
 which contradicts the fact that $\underset{p\in\mathcal{P}}{\sum}\frac{1}{p}=\infty$.

\end{proof}
\begin{cor}

For every $\eps>0$, there are infinitely many $n\in\bbN$ satisfying
\[
\underset{\underset{\mu\neq\lambda}{\lambda,\mu\in B\left(n,k\right)}}{\sum}a\left(\lambda-\mu\right)<\eps\,.
\]

\end{cor}

\begin{proof}

Fix $\eps>0$ and let $n$ and $k=k\left(n\right)$ be as follows
from Lemma 5 with $\nicefrac{\eps}{2}$. For $\lambda,\mu\in B\left(n,k\right)$
we write
\[
\lambda=j_{1}w_{n,k},\quad \mu=j_{2}w_{n,k},\qquad 1\leq j_{1}<j_{2}\leq n^{2}.
\]

Symmetry of $\left\{ a\left(m,m^{\prime}\right)\right\} $ implies
\[
\underset{\underset{\mu\neq\lambda}{\lambda,\mu\in B\left(k,n\right)}}{\sum}a\left(\lambda-\mu\right)=2\underset{j_{1}=1}{\sum^{n^{2}-1}}\underset{j_{2}=j_{1}+1}{\sum^{n^{2}}}a\left(\left(j_{2}-j_{1}\right)w_{n,k}\right)=
\]
\[
=2\underset{j=1}{\sum^{n^{2}-1}}\left(n^{2}-j\right)a\left(jw_{n,k}\right)\leq2n^{2}\underset{j=1}{\sum^{n^{2}}}a\left(jw_{n,k}\right)<\eps.
\]

\end{proof}
\begin{lem}

Given $\mathcal{S}\subset\bbT^{2}$ of positive measure, there exist
a constant $\gamma=\gamma\left(\mathcal{S}\right)>0$ for which the
following holds: For infinitely many $n\in\bbN$ there exists $1\leq k\leq n-1$
s.t. $\gamma$ is a lower Riesz bound \emph{(}in $L^{2}\left(\mathcal{S}\right)$\emph{)}
for $E\left(B\left(n.k\right)\right)$.

\end{lem}

\begin{proof}

The proof follows from Corollary 6 (see Lemma 7 in \cite{londner2014riesz}).

\end{proof}
\begin{lem}

Let $\gamma>0$, $\mathcal{S}\subset\bbT^{2}$ with $\left|\mathcal{S}\right|>0$,
and $B_{1},B_{2}\subset\bbN^{2}$ finite subsets s.t $\gamma$ is
a lower Riesz bound \emph{(}in $L^{2}\left(\mathcal{S}\right)$\emph{)}
for $E\left(B_{j}\right)$, $j=1,2$. Then for any $0<\gamma'<\gamma$
there exists $M\in\bbZ^{2}$ s.t. the system $E\left(B_{1}\cup\left(M+B_{2}\right)\right)$
has $\gamma'$ as a lower Riesz bound.

\end{lem}

\begin{proof}

See \cite{londner2014riesz}, Lemma 8.

\end{proof}

Now we are ready to finish the proof of Theorem 2.

Given $\mathcal{S}$ take $\gamma$ from Lemma 7 and denote by $\mathcal{N}$
the set of all pairs of natural numbers $\left(n,k\right)$ for which
$\gamma$ is a lower Riesz bound (in $L^{2}\left(\mathcal{S}\right)$)
for $E\left(B\left(n,k\right)\right)$. Define 
\[
\Lambda=\underset{\left(n,k\right)\in\mathcal{N}}{\bigcup}\left(M_{n,k}+B\left(n,k\right)\right)\,.
\]
Due to Lemma 8 we can choose, subsequently for every $\left(n,k\right)\in\mathcal{N}$,
a vector $M_{n,k}\in\bbZ^{2}$ s.t. $E\left(\Lambda\right)$ has lower
Riesz bound at least $\frac{\gamma}{2}$.

\section{Open problems}

Going over the proof one can easily see that both results extend to
the $d$ dimensional torus for any $d\in\bbN$. Yet a couple of questions
still remain unanswered.

\textbf{Problem 1.} What can be said about sets which contains arbitrarily
long GAP of size $N^{2}$ and step(s) length(s) at most $o\left(N\right)$. Does every set $\SS$ admit a Riesz sequence with such
property? or is there a set which rejects them all?

Notice that this is not known even in the one dimensional case.

~

The explicit construction in the proof of Theorem 2 yields a set $\Lambda$
which contains arbitrarily long rank 1 arithmetic progressions of
size $N^{2}$ and step length $N$.

\textbf{Problem 2.} Can one find an explicit construction as in the
proof of Theorem 2 with full rank GAPs?

\end{onehalfspace}
\end{document}